\documentclass[a4paper,11pt]{amsart}

\usepackage{amssymb}
\usepackage{amsmath}
\usepackage{latexsym}
\usepackage{eepic}
\usepackage{epsfig}
\usepackage{graphicx}
\usepackage[all,cmtip]{xy}

\usepackage{eufrak}
\usepackage{mathrsfs}
\usepackage[english]{babel}
\usepackage{color}
\usepackage[all,cmtip]{xy}

\DeclareMathAlphabet{\mathpzc}{OT1}{pzc}{m}{it}
\newtheorem{theorem}[equation]{Theorem}
\newtheorem{thm}[equation]{Theorem}
\newtheorem*{theorem*}{Theorem}
\newtheorem{theorem-definition}[equation]{Theorem-Definition}
\newtheorem{lemma-definition}[equation]{Lemma-Definition}
\newtheorem{definition-prop}[equation]{Proposition-Definition}
\newtheorem{corollary}[equation]{Corollary}
\newtheorem{prop}[equation]{Proposition}
\newtheorem*{prop*}{Proposition}
\newtheorem{lemma}[equation]{Lemma}
\newtheorem{cor}[equation]{Corollary}
\newtheorem{definition}[equation]{Definition}
\newtheorem*{definition*}{Definition}

\newtheorem*{conjecture*}{Conjecture}

\theoremstyle{definition}
\newtheorem{exam}[equation]{Example}

\newtheorem*{question*}{Question}
\newtheorem{remark}[equation]{Remark}

\newcommand{\Z}{\ensuremath{\mathbb{Z}}}
\newcommand{\Q}{\ensuremath{\mathbb{Q}}}

\newcommand{\R}{\ensuremath{\mathbb{R}}}
\newcommand{\C}{\ensuremath{\mathbb{C}}}

\newcommand{\A}{\ensuremath{\mathbb{A}}}

\newcommand{\cX}{\ensuremath{\mathscr{X}}}

\newcommand{\cA}{\ensuremath{\mathscr{A}}}

\newcommand{\cY}{\ensuremath{\mathscr{Y}}}

\renewcommand{\R}{\ensuremath{\mathbb{R}}}
\renewcommand{\C}{\ensuremath{\mathbb{C}}}

\renewcommand{\A}{\ensuremath{\mathbb{A}}}
\newcommand{\fX}{\ensuremath{\mathfrak{X}}}
\newcommand{\fY}{\ensuremath{\mathfrak{Y}}}

\renewcommand{\cA}{\ensuremath{\mathscr{A}}}

\renewcommand{\cY}{\ensuremath{\mathscr{Y}}}

\newcommand{\Spec}{\ensuremath{\mathrm{Spec}\,}}
\newcommand{\Spf}{\ensuremath{\mathrm{Spf}\,}}

\newcommand{\red}{\mathrm{red}}

\newcommand{\Gal}{\mathrm{Gal}}

\newcommand{\an}{\mathrm{an}}

\newcommand{\llbr}{[\negthinspace[}
\newcommand{\rrbr}{]\negthinspace]}
\newcommand{\llpar}{(\negthinspace(}
\newcommand{\rrpar}{)\negthinspace)}

\newcommand{\wt}{\mathrm{wt}}
\newcommand{\snc}{\mathrm{snc}}

\newcommand{\mld}{\mathrm{mld}}
\newcommand{\Sk}{\mathrm{Sk}}

\newcommand{\comp}[2]{\widehat{#1_{/#2}}}

\numberwithin{equation}{section}

\newcommand{\sss}{\vspace{5pt} \subsection*{ }\refstepcounter{equation}{{\bfseries(\theequation)}\ }}

\author[Johannes Nicaise]{Johannes Nicaise}
\address{Imperial College,
Department of Mathematics, South Kensington Campus,
London SW72AZ, UK, and KU Leuven, Department of Mathematics, Celestijnenlaan 200B, 3001 Heverlee, Belgium.} \email{j.nicaise@imperial.ac.uk}

\author{Chenyang Xu}
\address{Massachusetts Institute of Technology, 77 Massachusetts Ave., Cambridge, MA, USA}
\email{cyxu@math.mit.edu}

\author{Tony Yue Yu}
\address{Laboratoire de Math\'ematiques d'Orsay, Universit\'e Paris-Sud, 91405 Orsay, France}\email{yuyuetony@gmail.com}



\thanks{Johannes Nicaise is supported by the ERC Starting Grant MOTZETA (project 306610) of the European Research Council, and by long term structural funding (Methusalem
grant) of the Flemish Government. A part of the research leading to these results was carried out at the Freiburg Institute for Advanced Studies (FRIAS) with funding from the People Programme (Marie Curie Actions) of the European  Union's  Seventh Framework Programme (FP7/2007-2013) under REA grant agreement number 609305. Chenyang Xu is supported by the National Science Fund for Distinguished Young Scholars (11425101), ``Algebraic Geometry.'' Tony Yue Yu is supported by the Clay Mathematics Institute.}

\begin{document}
\title{The non-archimedean SYZ fibration}

\begin{abstract}
We construct non-archimedean SYZ fibrations for maximally degenerate Calabi-Yau varieties, and we show that they are affinoid torus  fibrations away from a codimension two subset of the base. This confirms a prediction by Kontsevich and Soibelman. We also give an explicit description of the induced integral affine structure on the base of the SYZ fibration. Our main technical tool is a study of the structure of minimal dlt-models along one-dimensional strata.
\end{abstract}

\maketitle

\section{Introduction}
\sss The theory of mirror symmetry emanated from string theory and has had a fundamental impact on algebraic geometry ever since the groundbreaking work of Candelas, de la Ossa, Green and Parkes \cite{COGP}. The mirror symmetry heuristic predicts that every complex Calabi-Yau manifold $X$ has a mirror partner $\check{X}$ of the same dimension whose complex geometry is equivalent, in a suitable sense, to the symplectic geometry of $X$, and vice versa. A celebrated application of these ideas was the prediction
of the numbers of rational curves of fixed degree (more precisely, Gromov-Witten invariants) of the quintic threefold in \cite{COGP} by means of period integral calculations on the mirror partner.
    An important challenge in the theory of mirror symmetry is to give an exact definition of what it means to be a mirror pair of Calabi-Yau manifolds, and to devise techniques to construct such pairs.

\sss In recent years, much progress has been made, in particular by Kontsevich--Soibelman \cite{KoSo2000,KoSo} and Gross--Siebert \cite{GS}. Both of these programs are based on a conjectural
 geometric explanation of mirror symmetry due to Strominger, Yau and Zaslow, known as the SYZ conjecture \cite{SYZ}.
  Since its appearance, the conjecture has been amended in certain ways; a common way to formulate it today is the following.
  Let $\mathcal{X}^{\ast}$ be a projective family of $n$-dimensional complex Calabi-Yau varieties
 over a punctured disk $\Delta^{\ast}$, and assume that this family is maximally degenerate. The latter condition means that
 the monodromy transformation on the degree $n$ cohomology of the general fiber $\mathcal{X}_t$ of $\mathcal{X}^{\ast}$ has a Jordan block of rank $n+1$. Then, up to rescaling the metrics, the family $\mathcal{X}_t$ is conjectured to converge in the Gromov-Hausdorff limit to an $n$-dimensional topological manifold $S$.
 Moreover, a general fiber $\mathcal{X}_t$
 should admit a fibration $\rho\colon\mathcal{X}_t\to S$, called an {\em SYZ} fibration, whose fibers are special Lagrangian tori in $\mathcal{X}_t$, except over a discriminant locus of codimension at least $2$ in the base $S$.
     The mirror partner of $\mathcal{X}_t$ can then be constructed by dualizing the torus fibration $\rho$ over the smooth locus and compactifying the result in an appropriate way (this involves deforming the dual fibration by so-called quantum corrections). We refer to the excellent survey paper \cite{Gro} for a more precise statement and additional background on the SYZ conjecture, as well as the Gross--Siebert program.

\sss The SYZ conjecture remains largely open, and is quite difficult even in basic cases; see for instance \cite{GW}. A fundamental insight of Kontsevich and Soibelman in \cite{KoSo} is that one should be able to construct a close analog of the SYZ fibration in the world of non-archimedean geometry, more precisely in the context of Berkovich spaces. Here, the base $S$ of the fibration arises as a so-called {\em skeleton} in the Berkovich analytification of the degeneration.  Let us emphasize that the non-archimedean SYZ fibration is not merely an analog of the conjectural structure in a different context; it can effectively be used to realize the original goal of constructing mirror partners over the complex numbers, since one can go back from the non-archimedean world to the complex world by means of non-archimedean GAGA and algebraization techniques.
 In the non-archimedean approach, the quantum corrections are provided by non-archimedean enumerative geometry and wall-crossing structures  \cite{KoSo,yu,yub,KY18}.
  The non-archimedean SYZ fibration induces an affine structure with singularities on the base $S$, and Kontsevich and Soibelman made the striking conjecture that this affine manifold should be related to the Gromov-Hausdorff limit of $\mathcal{X}$ (Conjecture 3 in \cite{KoSo}) -- see \cite{BoJo} for interesting results towards that conjecture.

\sss The aim of the present paper is to construct the non-archimedean SYZ fibration in general, and to prove some of its conjectural properties.
  This paves the way for a better understanding of the Gromov-Hausdorff limits and the SYZ conjecture.
 Our construction of the SYZ fibration builds upon the original work of Kontsevich and Soibelman and the relations with the Minimal Model Program discovered by the first two authors in \cite{NiXu}. This discovery has led to a surprising dictionary where the SYZ heuristic can be translated into precise predictions about the structure of minimal models, which can then be proven with techniques from the Minimal Model Program -- see for instance \cite{KoXu} and \cite{NiXu2}.
 Our main new result here is that the non-archimedean SYZ fibration is a smooth affinoid torus fibration away from a codimension two subset of the base (Theorem \ref{thm:main}), as implied by Conjectures 1 and 3 in \cite{KoSo}. This amounts to proving that
 minimal dlt models with reduced special fiber of Calabi-Yau varieties are snc along the one-dimensional strata of the special fiber (Theorem \ref{theo:dlt}), and have a toric structure along these strata (Proposition \ref{prop:toric}).

\subsection*{Acknowledgements} We are grateful to the referee for carefully reading the article and
 pointing out a mistake in the proof of Proposition \ref{prop:toric}. The first-named author is grateful to Mirko Mauri for a helpful discussion on Proposition \ref{prop:approx0}.

\subsection*{Preliminaries and notation}
\sss We fix an algebraically closed field $k$ of characteristic $0$ and we set $R=k\llbr t \rrbr$ and $K=k\llpar t \rrpar$.
 We also fix an algebraic closure $K^a$ of $K$.
We denote by $\mathrm{ord}_t$ the $t$-adic valuation on $K$
 and we define an absolute value $|\cdot|$ on $K$ by setting $|a|=\exp(-\mathrm{ord}_t a)$ for every $a\in K^{\ast}$. This turns $K$ into a complete non-archimedean field.
 The natural logarithm, inverse to the exponential $\exp$, will be denoted by $\ln$.

\sss We denote by $(\cdot)^{\an}$ the analytification functor from the category of $K$-schemes of finite type to Berkovich's category of $K$-analytic spaces.
 For every $R$-scheme $\cX$, we will denote by $\cX_k=\cX\times_R k$ and $\cX_K=\cX\times_R K$ its special and generic fiber.

\sss If $\cX$ is a Noetherian $R$-scheme and $C$ is a subscheme of $\cX_k$, then we will denote by $\comp{\cX}{C}$ the formal completion of $\cX$ along $C$. If $\cX$ is of finite type over $R$, then $\comp{\cX}{C}$ is formally of finite type over $R$ (or {\em special}, in the terminology of \cite{berk-vanish2}). That is, it has a finite cover by open formal subschemes of the form $\Spf(A)$ where $A$ is a quotient of a topological $R$-algebra of the form $R\{x_1,\ldots,x_m\}\llbr y_1,\ldots,y_n\rrbr$.
  Every Noetherian formal scheme $\fX$ has a unique maximal ideal of definition $\mathscr{I}$, consisting of all the topologically nilpotent elements in $\mathcal{O}_{\fX}$.
  The closed subscheme of $\fX$ defined by $\mathscr{I}$ will be denoted by $\fX_{\red}$. This construction induces a functor from the category of Noetherian formal schemes to the category of reduced Noetherian schemes. If $\fX$ is a scheme, then $\fX_{\red}$ is the maximal reduced closed subscheme of $\fX$.

\sss \label{sss:toric} A separated flat $R$-scheme of finite type $\cY$ is called {\em toric} if there exists a toric morphism of toric varieties $$Y\to \A^1_k=\Spec k[t]$$ such that $\cY$ is isomorphic to
 $Y\times_{k[t]}R$. Such a toric scheme can be defined by giving a finite fan $\Sigma$ of strongly convex rational polyhedral cones in $\R^n\times \R_{\geq 0}$ for some $n\geq 0$, together with a positive integer $\iota$; then one can take $Y$ to be the toric $k$-variety associated with $\Sigma$ and $Y\to \A^1_k$ to be the toric morphism induced by the morphism
 $$\R^n\times \R_{\geq 0}\to \R_{\geq 0}\colon (u,v)\mapsto \iota \cdot v.$$
This is a slight generalization of the standard definition of a torus embedding over $R$ used in \cite[IV.3]{KKMS}, which corresponds to the case $\iota=1$ of our definition.
 For instance, the $R$-scheme $\Spec k\llbr \sqrt{t} \rrbr$ is toric in our sense, but not in the sense of \cite[IV.3]{KKMS}.

\sss A Calabi-Yau variety over a field $F$ is a smooth, proper, geometrically connected $F$-scheme $X$ such that the canonical line bundle $\omega_X$ is trivial. In particular, our definition also includes abelian varieties. A {\em volume form} on a Calabi-Yau variety $X$ is a nowhere vanishing differential form of maximal degree, that is, a global generator for the canonical line bundle $\omega_X$.

\sss Let $\cX$ be a Noetherian scheme, and let $D$ be an effective divisor on $\cX$, with prime components $D_i,\,i\in I$.
   A {\em stratum} of $D$ is a connected component of the schematic intersection $D_J=\cap_{j\in J}D_j$, for some non-empty subset $J$ of $I$.
   An {\em open stratum} is a stratum $S$ minus the union of the prime components of $D$ that do not contain $S$.

\sss Let $X$ be a smooth and proper $K$-scheme. A {\em model} of $X$ is a proper flat $R$-scheme $\cX$ endowed with an isomorphism $\cX_K\to X$.
  An snc-model of $X$ is a regular model $\cX$  such that $\cX_k$ is a divisor with strict normal crossings. An snc-model is called {\em semistable} if $\cX_k$ is reduced. By the semistable reduction theorem \cite[Ch4\S3]{KKMS}, there exists a finite extension $K'$ of $K$ such that $X\times_K K'$ has a semistable snc-model over the integral closure of $R$ in $K'$; if $X$ is projective, then we can moreover obtain a projective semistable snc-model.

  A dlt-model of $X$ is a normal  model $\cX$ such that the pair $(\cX,\cX_{k,\red})$ is divisorially log terminal (dlt).
  We say that a dlt-model is {\em good} if every prime component of $\cX_{k,\red}$ is $\Q$-Cartier; this is slightly weaker than the usual condition that $\cX$ is $\Q$-factorial, but it is sufficient for our purposes.
  In particular, every snc-model is also a good dlt-model.
    A dlt-model $\cX$ is called {\em minimal} if
 the logarithmic relative canonical divisor $K_{\cX/R}+\cX_{k,\red}$ is semi-ample. When $X$ is Calabi-Yau, this is equivalent to saying that $K_{\cX/R}+\cX_{k,\red}$ is torsion; when, moreover, $\cX_k$ is reduced, then it is equivalent to saying that $K_{\cX/R}\sim 0$ (since $\mathcal{O}^{\ast}(X)=K^{\ast}$, a divisor $D$ on $\cX$ supported on $\cX_k$ is principal if and only if it is a multiple of $\cX_k$; so if some nonzero multiple of $D$ is principal and $\cX_k$ is reduced, then $D$ is itself principal).

 \sss We collect what is currently known about the existence of these models in the following theorem. We say that a smooth and proper $K$-scheme $X$ is {\em defined over a curve} if there exist a smooth $k$-curve $C$, a smooth $k$-variety $Y$, a proper morphism $Y\to C$ and an isomorphism of $k$-algebras $R\cong \widehat{\mathcal{O}}_{C,c}$ for some closed point $c$ of $C$ such that  $X$ is isomorphic to $Y\times_C \Spec(K)$. If $X$ is projective, then we can take $Y$ to be projective over $C$, because projectivity descends under arbitrary field extensions \cite[14.55]{GW-AG}.

\begin{thm}\label{thm:exist}
Let $X$ be a projective Calabi-Yau variety over $K$.
 \begin{enumerate}
  \item \label{it:defcurve} If $X$ is defined over a curve, then $X$ has a projective good minimal dlt-model over $R$.
 \item \label{it:notsscurve} If $X$ is defined over a curve or $\dim(X)\leq 2$, then there exists a finite extension $K'$ of $K$ such that $X$ has a projective good minimal dlt-model with reduced special fiber over the integral closure of $R$ in $K'$.
 \item \label{it:projss} If $X$ has a projective semistable snc-model over $R$, then $X$ has a projective minimal dlt-model with reduced special fiber over $R$.
 \item \label{it:notss}  There exists a finite extension $K'$ of $K$ such that $X$ has a projective minimal dlt-model with reduced special fiber over the integral closure of $R$ in $K'$.
     \end{enumerate}
\end{thm}
\begin{proof}
In case \eqref{it:defcurve}, we choose a smooth $k$-curve $C$, a smooth $k$-variety $Y$, a projective morphism $Y\to C$ and an isomorphism of $k$-algebras $R\cong \widehat{\mathcal{O}}_{C,c}$ for some closed point $c$ of $C$ such that $X$ is isomorphic to $Y\times_C \Spec(K)$. Shrinking $C$ around $c$, we can arrange that $Y$ is smooth over $C^o=C\setminus \{c\}$ with trivial relative canonical divisor. By running an MMP on $(Y,Y_{c,\red})$ over $C$, we find a new projective model $Y'$ over $C$ that is $\Q$-factorial and such that the pair $(Y',Y'_{c,\red})$ is dlt and $K_{Y'/C}+Y'_{c,\red}$ is relatively semi-ample (see \cite[2.3.6]{NiXu}). By base change to $\Spec(R)$, we obtain a projective good minimal dlt-model for $X$ over $R$. After a finite extension of $K$, we may furthermore assume that $Y_c$ is reduced, by the semistable reduction theorem; then $Y'_c$ is reduced, which proves \eqref{it:notsscurve} in the case where $X$ is defined over a curve. When $X$ has dimension at most $2$, then (after a finite extension of $K$) we may assume that $X$ has a projective semistable snc-model $\cX$ over $R$, and we can directly run an MMP on $\cX$ by \cite{kawamata}, producing a projective good minimal dlt-model with reduced special fiber over $R$.

Point \eqref{it:projss} follows from Theorem 2 in \cite{KNX} and its proof (unfortunately, that argument does not allow us to deduce the existence of a {\em good} minimal dlt-model). Point \eqref{it:notss} follows from \eqref{it:projss} and the semistable reduction theorem.
\end{proof}

\begin{remark}\label{rem:curve}
If $X$ is a projective Calabi-Yau variety over $K$, then it is expected that $X$ always has a projective good minimal dlt-model over $R$, and, up to a finite extension of $K$,
 also a projective good minimal dlt-model with reduced special fiber. Unfortunately, the necessary tools from the MMP have not yet been developed over the base ring $R$.
 For the applications to mirror symmetry, the case where $X$ is defined over a curve appears to be sufficient. Nevertheless, we will try to avoid this assumption as much as possible in the sequel, by means of some approximation arguments.
\end{remark}

\sss An {\em integral affine function} on an open subset of $\R^n$ is a continuous real-valued function that can locally be written as a degree one polynomial with coefficients in $\Z$. Beware that some authors, including \cite{KoSo}, allow a constant term in $\R$ in the degree one polynomial; our more restrictive definition is better suited for the purposes of this paper.

\section{Construction of the non-archimedean SYZ fibration}
\sss Let $X$ be a Calabi-Yau variety over $K$. The {\em essential skeleton} $\Sk(X)$ of $X$ was first defined by Kontsevich and Soibelman in \cite{KoSo}.
 The construction was then refined and generalized in \cite{MuNi}. Let $\omega$ be a volume form on $X$. Then one can attach to the pair $(X,\omega)$ a {\em weight function}
 $$\wt_{\omega}\colon X^{\an}\to \R\cup \{+\infty\}$$ that measures the degeneration of $(X,\omega)$ at $t=0$ along points of $X^{\an}$; see \cite[\S4.5]{MuNi}. The essential skeleton $\Sk(X)$ is the locus of points in $X^{\an}$ where $\wt_{\omega}$ reaches its minimal value. This definition only depends on $X$, and not on $\omega$, because multiplying $\omega$ with a scalar $\lambda\in K^{\ast}$ shifts the weight function by the constant $\mathrm{ord}_t\lambda$. The essential skeleton is a non-empty compact subspace of $X^{\an}$, which can be explicitly computed in the following way. Let $\cX$ be an snc-model of $X$, with special fiber $\cX_k=\sum_{i\in I}N_iE_i$. If we view $\omega$ as a rational section of the line bundle $\omega_{\cX/R}(\cX_{k,\red})$, then it defines a Cartier divisor on $\cX$ that we denote by $\mathrm{div}_{\cX}(\omega)$. It is supported on $\cX_k$ because $\omega$ is nowhere vanishing on $X$; thus we can write $\mathrm{div}_{\cX}(\omega)=\sum_{i\in I}\nu_i E_i$.
 If we denote by $\Delta(\cX)$ the dual intersection complex of $\cX_k$, then $\Sk(X)$ is canonically homeomorphic to the sub-$\Delta$-complex of $\Delta(\cX)$ spanned by the vertices corresponding to the components $E_i$ for which $\nu_i/N_i$ is minimal (see Theorem 3 in \cite[\S6.6]{KoSo} and Theorem 4.7.5 in \cite{MuNi}).
 In particular, $\Sk(X)$ is homeomorphic to a finite $\Delta$-complex of dimension $\leq \dim(X)$.

\sss Kontsevich and Soibelman postulated that $\Sk(X)$ should be the base of the non-archimedean SYZ fibration, but the definition of $\Sk(X)$ does not provide us with a map $X^{\an}\to \Sk(X)$. To construct such a map, we will use an alternative description of the essential skeleton that appeared in \cite{NiXu}. Let $\cX$ be a minimal dlt-model of $X$, and denote by $\cX^{\snc}$ the open subscheme of $\cX$ consisting of the points where $\cX$ is regular and $\cX_k$ has strict normal crossings. Then the dual intersection complex
 $\Delta(\cX^{\snc})$ of $\cX^{\snc}_k$ can be canonically embedded into $X^{\an}$ (see \cite[\S3]{MuNi}).
 It follows from \cite[3.3.3]{NiXu} that the image of this embedding is exactly the essential skeleton $\Sk(X)$. To be precise, it is assumed in the statement of \cite[3.3.3]{NiXu} that $\cX$ is $\Q$-factorial and defined over an algebraic curve, but these assumptions are not used in the proof. If the minimal dlt-model $\cX$ is good, we will now construct a continuous retraction  $\rho_{\cX}\colon X^{\an}\to \Sk(X)$ by generalizing the construction for snc-models in \cite[3.1.5]{MuNi}.

 \sss \label{sss:assum} Let $\cX$ be a good minimal dlt-model of $X$. We need to make the following technical assumption: {\em the strata of $\cX_k$ are precisely the log canonical centers of the pair $(\cX,\cX_{k,\red})$ that are contained in $\cX_k$.} By the definition of a dlt-model, every log canonical center of $(\cX,\cX_{k,\red})$ is a stratum. The converse implication
 is known when $\cX$ is defined over an algebraic curve \cite[4.16]{kollar}. We will prove in Corollary \ref{cor:approx} that it also holds when $\cX_k$ is reduced, which is
 the most important case for our purposes. We expect that the assumption is always satisfied, but the relevant parts of the Minimal Model Program have not been written down for $R$-schemes. In any case, if our technical assumption holds, we can proceed in the following way.

\sss \label{sss:rho}  Let $x$ be a point in $X^{\an}$ and let $\red_{\cX}(x)$ be its reduction on $\cX_k$ (see \cite[2.2.2]{MuNi}). Let $Z$ be the unique minimal stratum of $\cX_k$ that contains $\red_{\cX}(x)$. By our assumption \eqref{sss:assum}, $Z$ is a log canonical center of $(\cX,\cX_{k,\red})$. Then $Z\cap\cX^{\snc}$ is a non-empty stratum of $\cX^{\snc}_k$ by the definition of a dlt pair. Thus, it determines a unique face $\tau$ of the dual intersection complex $\Delta(\cX^{\snc})$.
  Let $E_1,\ldots,E_r$ be the prime components of $\cX_k$ that contain $Z$, and let $N_1,\ldots,N_r$ be their multiplicities in $\cX_k$.
   Then $E_1,\ldots,E_r$ correspond precisely to the vertices $v_1,\ldots,v_r$ of $\tau$.
  We choose a positive integer $m$ such that $mE_i$ is Cartier at the point $\red_{\cX}(x)$ for every $i$, and we choose a local equation $f_i=0$ for $mE_i$ at $\red_{\cX}(x)$.
  Then $\rho_{\cX}(x)$ is the point of the simplex $\tau$ with barycentric coordinates $$\alpha=\frac{1}{m}(-N_1\ln|f_1(x)|,\ldots,-N_r\ln|f_r(x)| )$$ with respect to the vertices $(v_1,\ldots,v_r)$. Under the embedding of $\Delta(\cX^{\snc})$ into $X^{\an}$, the point $\rho_{\cX}(x)$ corresponds to the monomial point represented by $(\cX,(E_1,\ldots,E_r),\xi)$ and the tuple $$\frac{1}{m}(-\ln|f_1(x)|,\ldots,-\ln|f_r(x)|),$$ in the terminology of \cite[2.4.5]{MuNi}. It is obvious that this definition does not depend on the choices of $m$ and the local equations $f_i$. It is also straightforward to check that  $\rho_{\cX}$ is continuous, and that it is a retraction onto $\Delta(\cX^{\snc})=\Sk(X)$.

\begin{definition}\label{defi:SYZ}
Let $X$ be a Calabi-Yau variety over $K$ and let $\cX$ be a good minimal dlt-model of $X$ that satisfies assumption \eqref{sss:assum}.
 Then we call the map $\rho_{\cX}\colon X^{\an}\to \Sk(X)$ constructed in \eqref{sss:rho} the {\em non-archimedean SYZ fibration} associated with $\cX$.
 \end{definition}

\sss \label{sss:intaff} Beware that, even though the subspace $\Sk(X)$ of $X^{\an}$ only depends on $X$, the $\Delta$-structure on $\Delta(\cX^{\snc})=\Sk(X)$ and the retraction $\rho_{\cX}\colon X^{\an}\to \Sk(X)$ depend on the choice of the (good) minimal dlt-model $\cX$; we will illustrate this in Example \ref{exam:K3} below. However, the essential skeleton $\Sk(X)$ does carry a canonical piecewise integral affine structure, which is induced by the embedding into the $K$-analytic space $X^{\an}$: see \cite[\S3.2]{MuNi}. If $\cX$ is a minimal dlt-model for $X$, then this piecewise integral affine structure coincides with the one induced by the $\Delta$-complex structure on $\Delta(\cX^{\snc})=\Sk(X)$, provided that the barycentric coordinates on the faces of $\Delta(\cX^{\snc})$ are weighted by the multiplicities of the prime components in $\cX_k$ as in \cite[3.2.1]{MuNi}.

\begin{exam}\label{exam:K3}
Let $X$ be a maximally degenerate $K3$ surface over $K$, and let $\cX$ be a good minimal dlt-model over $R$ with reduced special fiber.
 Then $\Sk(X)$ is a $2$-sphere, and $\Delta(\cX^{\snc})$ provides this sphere with a triangulation.
  Different choices of $\cX$ are related by {\em elementary modifications} (flops) of type 0, 1 or 2 \cite[pp.12-15]{friedman-morrison}. An elementary modification of type 0 does not affect the triangulation of $\Sk(X)$
 or the map $\rho_{\cX}$, because it only changes $\cX$ along a curve contained in a two-dimensional open stratum. An elementary modification of type 1 does not modify the triangulation of $\Sk(X)$ but it does alter the map $\rho_{\cX}$, because the points of $X^{\an}$ that specialize to the minus one curve that is flipped (but not to its intersection with a double curve) will be mapped to a different vertex of $\Sk(X)$. Finally, an elementary modification of type 2 flips an edge in the triangulation of $\Sk(X)$, but does not alter $\rho_{\cX}$ because $\rho_{\cX}$ is invariant under blow-ups of strata in snc-models (see Propositions 3.1.7 and 3.1.9 in \cite{MuNi}).

 Denote by $Z$ the set of vertices of $\Delta(\cX^{\snc})= \Sk(X)$.  It follows from our main result, Theorem \ref{thm:main}, that the fibers of $\rho_{\cX}$ over all the points of $\Sk(X)\setminus Z$ are $2$-dimensional affinoid tori; more precisely, the restriction of $\rho_{\cX}$ over $\Sk(X)\setminus Z$ is an affinoid torus fibration in the sense of \eqref{sss:torusfibration}. The vertices in $Z$ correspond to the prime components of $\cX_k$, and the fiber of $\rho_{\cX}$ over a vertex depends on the geometry of $\cX$ along the corresponding component $E$: it is isomorphic to the generic fiber of the formal completion of $\cX$ along the open stratum of $\cX_k$ whose closure is $E$.
\end{exam}

\begin{prop}\label{prop:retract}
Let $X$ be a projective Calabi-Yau variety over $K$. Then the essential skeleton $\Sk(X)$ is a strong deformation retract of $X^{\an}$. If $\cX$ is a good minimal dlt-model that satisfies the assumption in \eqref{sss:assum}, then $\rho_{\cX}$ is homotopic to the identity on $X^{\an}$ relative to $\Sk(X)$.
\end{prop}
\begin{proof}
It is shown in \cite[4.2.4]{NiXu} that $\Sk(X)$ is a strong deformation retract of $X^{\an}$. This implies that every continuous retraction $X^{\an}\to \Sk(X)$ is homotopic to the identity on $X^{\an}$ relative to $\Sk(X)$; in particular, this is true for the retraction $\rho_{\cX}$.
\end{proof}

\sss Let $X$ be a Calabi-Yau variety over $K$ of dimension $n$. We say that $X$ is {\em maximally degenerate} if $X$ has a semistable snc-model over $R$ and the essential skeleton $\Sk(X)$ has dimension $n$. This is the class of Calabi-Yau varieties where we expect the SYZ mirror symmetry picture to appear. If $X$ is projective, then the condition $\dim(\Sk(X))=n$ is equivalent to the property that, for any topological generator $\sigma$ of $\Gal(K^a/K)\cong \widehat{\mu}(k)$ and any prime number $\ell$, the action of $\sigma$ on the \'etale cohomology space
$$H^{n}_{\text{\'et}}(X\times_K K^a,\Q_\ell)$$
has a Jordan block of rank $n+1$, by \cite[4.2.4(4)]{NiXu}. If $X$ is maximally degenerate and projective, then $\Sk(X)$ is a closed pseudomanifold (see \cite[4.2.4(3)]{NiXu} -- in the statement of that result, one should add the assumption that $X$ has a semistable snc-model, like in \cite[4.1.7]{NiXu}).
  If we assume, moreover, that $X$ is geometrically simply connected and $h^{i,0}(X)=0$ for $0<i<n$, then it is expected that $\Sk(X)$ is homeomorphic to $S^n$.
  It is not difficult to prove that $\Sk(X)$ has the $\Q$-rational homology of $S^n$, and that its fundamental group has trivial profinite completion; see \cite[4.2.4(4)]{NiXu} and \cite[6.1.3(4)]{HaNi}. When $X$ is defined over a curve, the homeomorphic equivalence of $\Sk(X)$ with $S^n$ has been established in \cite{KoXu} when $n\leq 3$, and also when $n=4$ and $X$ has a minimal semistable snc-model. These results can be extended in the following way.

  \begin{lemma}\label{lemma:approx0}
  Let $X$ be a projective Calabi-Yau variety over $K$ of dimension $n$.
   Then there exist a smooth pointed $k$-curve $(S,s)$, a projective flat morphism of $k$-schemes $f\colon \cY\to S$ of relative dimension $n$
   and an isomorphism of $k$-algebras $\widehat{\mathcal{O}}_{S,s}\cong R$ with the following properties:
   \begin{enumerate}
   \item the fibers of $f$ are geometrically connected;
    \item $f$ is smooth over $S\setminus \{s\}$, with trivial relative canonical bundle;
    \item the essential skeleta $\Sk(X)$ and $\Sk(\cY_K)$ are homeomorphic;
    \item if $X$ is geometrically simply connected then the same holds for the the fibers of $f$ over $S\setminus \{s\}$;
    \item if $h^{i,0}(X)=0$ for $0<i<n$, then the same holds for the the fibers of $f$ over $S\setminus \{s\}$.
       \end{enumerate}
       If we fix a model $\cX$ for $X$ over $R$ and a positive integer $N$, then we can moreover arrange that the $R$-schemes $\cX\times_R R/(t^N)$ and
       $\cY\times_S \Spec(R/t^N)$ are isomorphic.
  \end{lemma}
  \begin{proof}
One can construct $(S,s)$ and $\cY$ by means of a standard argument based on spreading out and Greenberg approximation; see for instance \cite[5.1.2]{MuNi} and the proof of \cite[4.2.4]{NiXu}. The identification of the essential skeleta of $X$ and $\cY_K$ follows from \cite[4.2.3]{NiXu}.
 Geometric simple connectedness and the vanishing of Hodge numbers carry over to the smooth fibers of the spreading out, by Grothendieck's specialization theorem for \'etale fundamental groups \cite[X.3.9]{sga1} and the invariance of Hodge numbers in smooth and proper families \cite[5.5]{deligne} (in fact, for our purposes, semicontinuity suffices).
  \end{proof}

  \begin{prop}\label{prop:approx0}
    Let $X$ be a projective Calabi-Yau variety over $K$ of dimension $n$.
     \begin{enumerate}
\item   \label{it:simply}  Assume that $X$ has a semistable snc-model.
If $X$ is geometrically simply connected, then $\Sk(X)$ is simply connected.
\item  \label{it:sphere}   Assume that $X$ is maximally degenerate and geometrically simply connected, and that $h^{i,0}(X)=0$ for $0<i<n$. Then $\Sk(X)$ is homeomorphic to $S^n$ when $n\leq 3$, and also when $n=4$ and $X$ has a semistable snc-model $\cX$ with $K_{\cX/R}\sim 0$.
    \end{enumerate}
  \end{prop}
  \begin{proof}
  \eqref{it:simply} Using Lemma \ref{lemma:approx0}, we can replace $X$ by a degenerating projective flat family of Calabi-Yau varieties $f\colon \cY\to S$ over a smooth pointed $k$-curve $(S,s)$ such that $\cY$ is smooth over $k$, the smooth geometric fibers of $f$ have trivial  \'etale fundamental group, and the fiber $\cY_s$ over $s$
   is a reduced strict normal crossings divisor.
      By the Lefschetz principle, we may assume that $k=\C$. It follows from the Beauville-Bogomolov decomposition theorem that
  a complex Calabi-Yau variety is simply connected if and only if it has no nontrivial finite \'etale covers. Thus the smooth closed fibers of $f$ are simply connected.
   Now we can copy the argument in \S34 of \cite{KoXu} to deduce that the dual intersection complex of $\cY_s$ is simply connected; this dual intersection complex is homotopy equivalent to the essential skeleton $\Sk(\cY_K)$, by \cite[3.2.8]{NiXu}.

  \eqref{it:sphere} We again reduce to the case where $X$ is defined over a complex curve, which has been solved in Proposition 8 in \cite{KoXu}.
  The reduction in the case $n=3$ follows once more from Lemma \ref{lemma:approx0} and the argument in \eqref{it:simply}.
  Now assume that $X$ has a semistable snc-model $\cX$ with $K_{\cX/R}\sim 0$, and let $\cY$ be an algebraic approximation of this model as in Lemma \ref{lemma:approx0}, with $N\geq 2$. Then $\cY$ is smooth over $k$ and $\cY_k$ is a reduced strict normal crossings divisor, since these properties can be checked on the reduction modulo $t^2$.
   The same holds for minimality of such a model, because it is equivalent to the triviality of the
 logarithmic relative canonical line bundle on the special fiber, and the induced logarithmic structure on the special fiber only depends on reduction of $\cY$ modulo $t^2$ (this result is due to Illusie; a proof can be found in \cite[A.4]{nakayama}). Thus $\cY$ satisfies the conditions of Proposition 8 in \cite{KoXu}.
  \end{proof}

\section{Affinoid torus fibrations}
\sss Let $X$ be a maximally degenerate Calabi-Yau variety and let $\cX$ be a good minimal dlt-model of $X$ with reduced special fiber. Then we will see in Corollary \ref{cor:dlt} that $\cX$ satisfies assumption \eqref{sss:assum}, so that it gives rise to a non-archimedean SYZ fibration $\rho_{\cX}\colon X^{\an}\to \Sk(X)$ in the sense of Definition \ref{defi:SYZ}.
 The principal aim of this article is to study the fibers of $\rho_{\cX}$. In the classical SYZ conjecture, the fibers of the SYZ fibration are expected to be special Lagrangian tori away from a codimension two subset of the base. We will now present the corresponding structure in non-archimedean geometry, which was introduced in \cite[\S4.1]{KoSo}.

\sss Let $n$ be a positive integer, and let $T$ be a split algebraic $K$-torus of dimension $n$ with character module $M$ and cocharacter module $N=M^{\vee}$.  We define the {\em tropicalization map} of $T$ by
$$\rho_T \colon T^{\an}\to N_{\R}\colon x\mapsto (M\to \R\colon m \mapsto -\ln|m(x)|).$$ This map is continuous, and its fibers are (not necessarily strictly) $K$-affinoid tori.
 The tropicalization map $\rho_T$ has a canonical continuous section $s\colon N_{\R}\to T^{\an}$ that maps each $n\in N_{\R}$ to the Gauss point of the
 affinoid torus $\rho^{-1}_{T}(n)$. The image of $s$ is called the canonical skeleton of $T$, and denoted by $\Delta(T)$. The map $s$ induces a homeomorphism $N_{\R}\to \Delta(T)$, which we will use to tacitly identify $\Delta(T)$ with $N_{\R}$.

\sss \label{sss:torusfibration} Let $Y$ be a $K$-analytic space, let $B$ be a topological space and let $f\colon Y\to B$ be a continuous map. Then we say that $f$ is an $n$-dimensional {\em affinoid torus  fibration} if we can cover $B$ by open subsets $U$ such that there exist an open subset $V$ of $N_{\R}\cong \R^n$ and a commutative diagram
$$\xymatrix{
f^{-1}(U)\ar[r] \ar[d]_{f} & \rho_T^{-1}(V) \ar[d]^{\rho_T}
\\ U \ar[r] & V
}$$
where the upper horizontal map is an isomorphism of $K$-analytic spaces and the lower horizontal map is a homeomorphism.

\sss If $f\colon Y\to B$ is an $n$-dimensional affinoid torus  fibration, then $f$ induces an integral affine structure on the base $B$ \cite[\S4.1]{KoSo}. For every open $U$ in $B$ as in the definition, and every invertible analytic function $h$ on $f^{-1}(U)$, the absolute value of $h$ is constant on the fibers of $f$ by the maximum modulus principle. Thus $h$ induces a continuous function
$|h|\colon U\to \R_{>0}$.
 The integral affine functions on $U$ are, by definition, the functions of the form $-\ln|h|$. If $U$ is connected, then it is proven in Theorem 1 of \cite[\S4.1]{KoSo} that under the homeomorphism $U\to V$, the ring of integral affine functions on $U$ is identified with the ring of polynomial functions of degree one with $\Z$-coefficients on $V\subset N_{\R}$, so that this construction indeed defines an integral affine structure on $B$ (to be precise, in \cite{KoSo} the authors consider affine functions with constant term in $\R$, rather than $\Z$, but since $K$ is discretely valued in our case, we get a slightly stronger result).

\begin{exam}\label{exam:toric}
  We use the tropicalization map to identify the canonical skeleton $\Delta(T)$ with $N_{\R}$. We denote by $C$ the open cone $(N_{\R}\times \R_{> 0})\cup\{0\}$ in $N_{\R}\oplus \R$.
 Let $\Sigma$ be a locally finite fan of strongly convex rational polyhedral cones in $C$.
  We denote by $\Sigma_1$ the rational polyhedral complex in $N_{\R}$ obtained by intersecting the cones in $\Sigma$ with $N_{\R}\times \{1\}$.
 Consider the torus embedding $T\to \cX$ over $R$
associated with $\Sigma$ as in \cite[1.13]{kunnemann}. The $R$-scheme $\cX$ is separated and locally of finite type, and it is quasi-compact if and only if $\Sigma$ is finite.
 Since $\Sigma$ is supported in $C$, the generic fiber of $\cX$ is canonically isomorphic to the split $K$-torus $T$.
 Assume that $\cX$ is regular; this is equivalent to the property that the fan $\Sigma$ is simple, and it implies that the special fiber $\cX_k$ is a strict normal crossings divisor. Denote by $\fX$ the formal $t$-adic completion of $\cX$. The generic fiber $\fX_\eta$ is a $K$-analytic space endowed with a natural injective morphism of $K$-analytic spaces $i:\fX_\eta\to T^{\an}$. The morphism $i$ embeds $\fX_\eta$ as an analytic domain in $T^{\an}$.

 The construction of the Berkovich skeleton $\Sk(\cX)$ and the retraction map $\rho_{\cX}$ in \cite[\S3]{MuNi} are local on $\cX$, so that they extend immediately to schemes that are locally of finite type. This yields a canonical embedding of the dual intersection complex $\Delta(\cX)$ of $\cX_k$ into $\fX_\eta$. The image of this embedding is called the Berkovich skeleton of $\cX$. The embedding has a canonical retraction $\rho_{\cX}:\fX_\eta\to \Delta(\cX)$.
   It follows directly from the definitions that $\Delta(\cX)$ is contained in $\Delta(T)=N_{\R}$ and coincides with the support of  $\Sigma_1$.
 In particular, if $\Sigma$ is a subdivision of $C$, then $\Delta(\cX)=\Delta(T)$. Moreover, the $\Delta$-structure on $\Delta(\cX)$ is precisely the polyhedral decomposition $\Sigma_1$.
 We have $\fX_\eta=\rho_{T}^{-1}(|\Sigma_1|)$, and the retraction map $\rho_{\cX}$ is the restriction of $\rho_T$ to $\fX_\eta$.
\end{exam}

\sss As a first application, let us discuss the case of abelian varieties. Let $A$ be an abelian $K$-variety of dimension $n$, and denote by $\cA$ its N\'eron model. Then Berkovich has constructed in \cite[\S6.5]{berk-book}
a {\em canonical skeleton} $\Delta(A)$ in $A^{\an}$, together with a continuous retraction $\rho_{A}\colon A^{\an}\to \Delta(A)$, via the theory of non-archimedean uniformization.
 The dimension of $\Delta(A)$ is equal to the toric rank of $\cA^o_k$ (the dimension of the maximal subtorus).
 Let us make this construction more precise in the maximally degenerate case. Assume that $A$ has purely toric reduction, that is, $\cA^o_k$ is a torus. Let $e$ be the identity point on $A$. Then the universal pointed covering space of $(A,e)$ (with respect to the Berkovich topology) is isomorphic to the analytification of a split $n$-dimensional $K$-torus $T$. The kernel $L$ of the morphism $\pi\colon T^{an}\to A^{\an}$ is a lattice in $T(K)$ (called the period lattice), and the image $\rho_T(L)$ of $L$ in $N_{\R}$ is a lattice of rank $n$. By definition, the canonical skeleton $\Delta(A)$ is the image of $\Delta(T)$ under the map $\pi$. Moreover, we have  a Cartesian diagram of topological spaces
 $$\xymatrix{
 T^{\an}\ar[r]^{\rho_T} \ar[d]_{\pi} & N_\R \ar[d]
 \\ A^{\an} \ar[r]^-{\rho_A} & N_{\R}/\rho_T(L)
 }$$
 such that $\rho_A$ sends $\Delta(A)$ homeomorphically onto $N_{\R}/\rho_T(L)$.
  In particular, $\Delta(A)$ is a real torus of dimension $n$, $\rho_A$ is an $n$-dimensional torus fibration, and the induced integral
  affine structure on $\Delta(A)$ coincides with the quotient structure on $N_{\R}/\rho_T(L)$.

\sss  If $A$ has purely toric reduction, then we can interpret $\rho_A$ as a non-archimedean SYZ fibration by means of the theory of Mumford models \cite{mumford}
 and the refinements of Mumford's construction given in \cite{kunnemann}. We say that a model $\mathscr{P}$ of $A$ is a {\em K\"unnemann-Mumford model}
  if it is a regular model that arises through the construction in the proof of \cite[3.5]{kunnemann}.

\begin{prop}\label{prop:abelian}
Let $A$ be an abelian $K$-variety of dimension $n$.  Then the essential skeleton $\Sk(A)$ of $A$ coincides with Berkovich's canonical skeleton $\Delta(A)$.
    If $A$ has semi-abelian reduction and $\mathscr{P}$ is a K\"unnemann-Mumford model of $A$ over $R$, then $\mathscr{P}$ is a good minimal dlt-model that satisfies assumption \eqref{sss:assum}. If $A$ has purely toric reduction, then the non-archimedean SYZ fibration $\rho_{\mathscr{P}}$ coincides with Berkovich's canonical retraction $\rho_A$. In particular, $\rho_{\mathscr{P}}$ is an $n$-dimensional affinoid torus  fibration.
\end{prop}
\begin{proof}
The equality $\Delta(A)=\Sk(A)$ is proven in \cite[4.3.2]{HaNi}.
 Let $\mathscr{P}$ be a K\"unnemann-Mumford model for $A$ over $R$. Then, by definition, $\mathscr{P}$ is an snc-model, and thus certainly good and dlt. It is shown in \cite[5.1.7]{HaNi} that $\mathscr{P}$ is minimal.

 Let $\widetilde{\mathscr{P}}$ be a regular relatively complete model of $T$ as in \cite[2.11]{kunnemann} such that the formal $t$-adic completion of
 $\mathscr{P}$ arises as a quotient of the formal $t$-adic completion of $\widetilde{\mathscr{P}}$ under an action of the period lattice. Then, by construction,  $\widetilde{\mathscr{P}}$ is a torus embedding of $T$ over $R$, and we have a commutative diagram
$$
\xymatrix{
T^{\an} \ar[d]_{\pi} \ar@/^/[r]^{\rho_T} \ar@/_/[r]_{\rho_{\widetilde{\mathscr{P}}}} & \Delta(T) \ar[d] \\
A^{\an} \ar@/^/[r]^{\rho_A} \ar@/_/[r]_{\rho_{\mathscr{P}}} &\Delta(A).}
$$
  Thus in order to prove that $\rho_{\mathscr{P}}=\rho_A$, it suffices to observe that
 $\rho_{\widetilde{\mathscr{P}}}=\rho_T$ by Example \ref{exam:toric}.
\end{proof}

\begin{remark}
A refinement of the proof shows that the equality $\rho_A=\rho_{\mathscr{P}}$ remains valid if we only assume that $A$ has semi-abelian reduction; then the non-archimedean
uniformization of $A$ takes the form $\pi:E^{\an}\to A^{\an}$, where $E$ is an extension of an abelian $K$-variety $B$ with good reduction by a split $K$-torus $T$. The dimension of $T$ is precisely the toric rank of $\cA^o_k$, the identity component of the special fiber of the N\'eron model of $A$. The K{\"u}nnemann-Mumford construction produces a relatively complete model $\widetilde{\mathscr{P}}$ of $E$ that is a Zariski-locally trivial fibration in torus embeddings over the N\'eron model of $B$. Since we do not need this generalization in this paper, we omit the details.
\end{remark}

\section{One-dimensional strata of minimal dlt-models}
\sss The aim of this section is to show that good minimal dlt-models with reduced special fibers of Calabi-Yau varieties over $K$ are snc along their one-dimensional strata (in fact, we will prove a more general result -- see Theorem \ref{theo:dlt} and Corollary \ref{cor:dlt}).
 A technical complication is that
the full machinery of the MMP has only been written down for objects of finite type over a field. To circumvent this problem,
we will first prove an approximation result (Proposition \ref{prop:approx}) that will allow us to reduce to that case.

\begin{lemma}\label{lemm:term}
Let $\cX$ be a normal $R$-scheme and let $D$ be a reduced effective divisor on $\cX$ such that $D$ contains the singular locus of $\cX$ and such that the pair $(\cX,D)$ is dlt. Assume that  $K_{\cX/R}+D$ is Cartier.
 Then $\cX$ is terminal; in particular, it is regular in codimension two.
\end{lemma}
\begin{proof}
By the definition of a dlt-pair, the scheme $\cX$ is regular at the generic point of every stratum of $D$, and at all the other points $x\in \cX$,
       the minimal log discrepancy $\mld_x(\cX,D)$ is positive. Since $K_{\cX/R}+D$ is Cartier,
  $\mld_x(\cX,D)$ is an integer, and therefore at least $1$.
  The inequality
  $$ \mld_x(\cX,0)>\mld_x(\cX,D)\geq 1$$ now implies that $\cX$ is terminal.
 In particular, it is regular in codimension two.
\end{proof}

\begin{prop}\label{prop:approx}
Let $X$ be a Calabi-Yau variety over $K$ and let $\cX$ be a good dlt-model of $X$ over $R$ such that $\cX_k$ is reduced.
 Assume that $K_{\cX/R}\sim 0$.
 Let $N$ be a fixed positive integer. Then we can find a smooth pointed $k$-curve $(S,s)$ and
a normal proper flat $S$-scheme $\cY$ such that the following properties hold:
\begin{enumerate}
\item \label{it:1} there exist
an isomorphism of $k$-algebras $\widehat{\mathcal{O}}_{S,s}\cong R$ and an isomorphism of $R$-schemes $$\cX\times_R R/(t^N)\to \cY\times_S \Spec(R/t^N);$$
\item \label{it:2} the morphism $\cY\to S$ has geometrically connected fibers, and its restriction over $S\setminus \{s\}$ is smooth with trivial relative canonical line bundle;
\item the pair $(\cY,\cY_s)$ is dlt, every prime component of $\cY_s$ is $\Q$-Cartier, and $K_{\cY/S}\sim 0$.
\end{enumerate}
\end{prop}
\begin{proof}
One can construct $(S,s)$ and $\cY$ satisfying \eqref{it:1} and \eqref{it:2} as in Lemma \ref{lemma:approx0} (note that normality of $\cY$ automatically follows from the fact that $\cY\setminus \cY_s$ is normal and $\cY_s\cong \cX_k$ is reduced).
 If $N$ is at least $2$, then for every point $x$ of $\cX_k\cong \cY_s$, the model $\cX$ is regular at $x$ if and only if $\cY$ is regular at $x$.
 Thus the pair $(\cY,\cY_s)$ is snc at all the points of $\cY_s$ where $(\cX,\cX_k)$ is snc.
  Taking $N$ sufficiently large, we can arrange that every prime component $E$ of $\cY_s$ is $\Q$-Cartier in $\cY$. More precisely, let $x$ be a point of $\cX_k$ and let $m_x$ be the Cartier index of $E$ in $\cX$ at $x$. Let $f$ be a local generator for the ideal sheaf $\mathcal{O}_{\cX}(-m_xE)$ at $x$. Assume that $N>m_x$ and let $g$ be any element of $\mathcal{O}_{\cY,x}$ that is congruent to $f$ modulo $t^{N}$. Obviously, $g$ cannot vanish at any other component of $\cY_s$, because $t$ vanishes along each of these components and $f$ does not. On the other hand, $f$ divides $t^{m_x}$ in $\mathcal{O}_{\cX,x}$, so that $g$ divides $t^{m_x}$ in $\mathcal{O}_{\cY,y}$ since $N>m_x$. Thus the zero locus of $g$ is supported in $\cY_s$, which means that $g=0$ is a local equation for $m_x E$ in $\cY$ at $x$.
 From now on, we assume that $N$ has been chosen large enough to guarantee that $N\geq 2$ and every prime component of $\cY_s$ is $\Q$-Cartier.

 Let $E$ be a prime component of $\cX_k$, denote by $\widetilde{E}$ its normalization, and let $\Delta$ be the pullback of the $\Q$-Cartier divisor $\cX_{k}-E$ to $\widetilde{E}$.  The scheme $\cX$ is regular in codimension two by Lemma \ref{lemm:term}. It follows that the different $\mathrm{Diff}_{\widetilde{E}}(\cX_{k}-E)$ coincides with $\Delta$. Thus the pair $(\widetilde{E},\Delta)$ is dlt by adjunction \cite[4.8]{kollar}, using the same reasoning as in the proof of \cite[4.16.4]{kollar} (except that we have not yet established the normality of $E$).
 Since $N\geq 2$, the scheme $\cY$ is regular in codimension two, as well; since it is of finite type over $k$, we can apply inversion of adjunction \cite[4.9]{kollar} to deduce that $(\cY,\cY_{s})$ is log canonical on a neighbourhood of $E$, and that the log canonical centers of $(\cY,\cY_{s})$ contained in $E$ are precisely the images of the  log canonical centers of $(\widetilde{E},\Delta)$. At the generic point of such a log canonical center, the pair $(\cY,\cY_{s})$ is snc because the same holds for $(\cX,\cX_{k})$.
 Varying $E$, we obtain that $(\cY,\cY_{s})$ is dlt. This implies that every stratum of $\cX_k\cong \cY_{s}$ is normal \cite[4.16]{kollar}; thus, in retrospect, we see that $\widetilde{E}=E$.
 \end{proof}
\begin{cor}\label{cor:approx}
Let $X$ be a Calabi-Yau variety over $K$ and let $\cX$ be a good dlt-model of $X$ over $R$ such that $\cX_k$ is reduced.
 Assume that $K_{\cX/R}\sim 0$.
 Then every stratum of $\cX_k$ is normal, and the strata of $\cX_k$ are precisely the log canonical centers of the pair $(\cX,\cX_{k})$ contained in $\cX_k$.
\end{cor}
\begin{proof}
In the proof of Proposition \ref{prop:approx}, we have constructed a dlt pair $(\cY,\cY_s)$ with $\cY$ of finite type over $k$ such that there exists an isomorphism of $k$-schemes
$\cX_k\to \cY_s$ that identifies the log canonical centers of $(\cX,\cX_k)$ contained in $\cX_k$ with those of $(\cY,\cY_s)$ contained in $\cY_s$. Thus the result follows from the corresponding properties of $(\cY,\cY_s)$ proven in \cite[4.16]{kollar}.
\end{proof}

\begin{theorem}\label{theo:dlt}
 Let $\cX$ be a normal separated $k$-scheme of finite type. Let $D$ be a reduced effective divisor on $\cX$ such that the pair $(\cX,D)$ is dlt and the divisor $K_{\cX}+D$ is Cartier.  Assume also that all the prime components of $D$ are $\Q$-Cartier.
 Let $C$ be a one-dimensional stratum of $D$. Then, on an open neighbourhood of $C$, the scheme $\cX$ is regular and $D$ is a divisor with strict normal crossings.
\end{theorem}
\begin{proof}
 We will argue by induction on the dimension of $\cX$. The case $\dim(\cX)=1$ follows at once from the fact that all strata of dlt pairs are normal \cite[4.16(2)]{kollar}. Thus we may assume that $\dim(\cX)\geq 2$, and that the result holds for pairs of strictly lower dimension.

Let $x$ be a point on $C$. We claim that every prime divisor in $D$ that contains $C$ is Cartier at $x$.
Assuming the claim for now, it follows that $C$ is a local complete intersection at $x$, and thus reduced because it is generically reduced (the pair $(\cX,D)$ is snc at the generic point of $C$). Now it follows from \cite[4.16(2)]{kollar} that $C$ is normal, and thus regular since it is of dimension one. But $C$ is defined by the local equations at $x$ of the prime components of $D$ that contain $C$; these local equations form a regular sequence, again by \cite[4.16(2)]{kollar}. We conclude that locally at $x$, the scheme $\cX$ is regular and $D$ is a strict normal crossings divisor.

 Thus it suffices to prove our claim.
 We may assume that $x$ is not a zero-dimensional stratum of $D$, since at such points, the pair $(\cX,D)$ is snc by the definition of a dlt pair.
    Let $E$ be a prime divisor in $D$ that contains $C$.  Let $F_1,\ldots,F_r$ be the non-empty intersections of $E$ with the other components of $D$, and set $\Delta=F_1+\ldots+F_r$.
  Then the pair $(E,\Delta)$ is dlt, and
  $$K_{E}+\Delta=(K_{\cX}+D)|_{E}$$ is Cartier (see Proposition~4.5 and Claim~4.16.4 in \cite{kollar}). By the induction hypothesis, $E$ is regular at $x$.

Let $m\geq 1$ be the index of $E$ at $x$, that is, the smallest positive integer such that $mE$ is Cartier at $x$.
  Working locally around $x$, we may assume that $E$ is regular and that $\mathcal{O}_{\cX}(mE)$ is a trivial line bundle.
The choice of a trivialisation determines a ramified $\mu_m$-cover
 $h\colon \widetilde{\cX}\to \cX$ defined by
 $$\widetilde{\cX}=\Spec_{\cX}\bigoplus_{a=0}^{m}\mathcal{O}_{\cX}(-aE).$$
  Here $\mathcal{O}_{\cX}(-aE)$ is the rank one reflexive sheaf associated with the Weil divisor $-aE$. This is the so-called {\em index one cover} of the pair $(\cX,E)$ at the point $x$; see \cite[2.52]{kollar-mori} for details. The morphism $h$ is \'etale over all the points where $E$ is Cartier; in particular, it is \'etale over all the codimension one points of $E$, since $\cX$ is regular in codimension two by Lemma \ref{lemm:term}.
   The minimality of $m$ implies that the inverse image of $x$ in $\widetilde{\cX}$ consists of a unique point, which we denote by $\widetilde{x}$.

   We write $\widetilde{E}$ for the inverse image of $E$ on $\widetilde{\cX}$, and $\widetilde{D}$ for the inverse image of the divisor $D$. By \cite[5.20]{kollar-mori}, the pair $(\widetilde{\cX},\widetilde{D})$ is log canonical, and $\mld_{\widetilde{x}}(\widetilde{\cX},\widetilde{D})$ is positive.
    Since we chose $x$ on a one-dimensional stratum $C$, the divisor $D$ has $\dim(\cX)-1$ prime components that pass through $x$.
    This implies that $\widetilde{E}$ is unibranch at $\widetilde{x}$. Otherwise, \'etale-locally around $\widetilde{x}$, the divisor $\widetilde{D}$ would have at least $\dim(\cX)$ components passing through $\widetilde{x}$, and $\widetilde{x}$ would be their intersection; but this implies that $\widetilde{x}$ is a log canonical center of $(\widetilde{X},\widetilde{D})$, by \cite[4.41(2)]{kollar}, contradicting the positivity of $\mld_{\widetilde{x}}(\widetilde{\cX},\widetilde{D})$.

 We denote by $\widetilde{E}'$ the normalization of $\widetilde{E}$. Since $\widetilde{E}$ is unibranch at $\widetilde{x}$, there is a unique point $\widetilde{x}'$ on $\widetilde{E}'$ that lies above $\widetilde{x}\in \widetilde{E}$.
   We have already observed that the morphism $\widetilde{E}\to E$ induced by $h$ is \'etale in codimension one; then the normality of $E$ implies that $\widetilde{E}$ is normal in codimension one. Thus $\widetilde{E}'\to E$ is also \'etale in codimension one.
     Since $E$ is regular, the purity of the branch locus now implies that the finite morphism $\widetilde{E}'\to E$ is \'etale at $\widetilde{x}'$; but $\widetilde{x}'$ is the unique point that lies above $x\in E$, so that $\widetilde{E}'\to E$, and hence $\widetilde{E}\to E$, are isomorphisms. We finally conclude that $m=1$, so that $E$ is Cartier at $x$.
\end{proof}

\begin{corollary}\label{cor:dlt}
Let $X$ be a Calabi-Yau variety over $K$, and let $\cX$ be a good minimal dlt-model for $X$ over $R$.
Assume that the special fiber $\cX_k$ is reduced.
 Let $C$ be a one-dimensional stratum of $\cX_k$.
 Then, on an open neighbourhood of $C$, the scheme $\cX$ is regular and $\cX_k$ is a divisor with strict normal crossings.
\end{corollary}
\begin{proof}
 The property that the pair $(\cX,\cX_{k,\red})$ is snc at a point of $\cX_k$ only depends on the reduction of $\cX$ modulo $t^2$.
 Thus by means of the approximation result in Proposition \ref{prop:approx}, we can reduce to the case where the model $\cX$ is defined over a smooth algebraic $k$-curve; then the result follows from Theorem \ref{theo:dlt}.
\end{proof}

\section{Toric structure of snc-models along one-dimensional strata}
\sss Let $\cX$ be a regular flat $R$-scheme such that $\cX_k$ is a strict normal crossings divisor.
 We write $\cX_k=\sum_{i\in I}N_iE_i$, where $E_i,\,i\in I$ are the prime divisors in $\cX_k$ and the numbers $N_i$ are their multiplicities. By the definition of a strict normal
 crossings divisor, every stratum of $\cX_k$ is a regular $k$-scheme.
 Let $C$ be a stratum of $\cX_k$. We say that $\cX$ is {\em toric} along $C$ if there exist a regular toric $R$-scheme $\cY$ and a stratum $D$ of $\cY_k$ such that $\cY_k$ is a strict normal crossings divisor and the formal $R$-schemes $\comp{\cX}{C}$ and $\comp{\cY}{D}$ are isomorphic.

\sss \label{sss:CY} Now let $C$ be a one-dimensional stratum of $\cX_k$ that is proper over $k$. Let $E_j,\,j\in J$ be the prime components of $\cX_k$ that contain $C$. For every $j\in J$, we set
 $$b_j=\mathrm{deg}\mathcal{O}_C(-E_j)=-(C\cdot E_j).$$
 We write $C^o=C\setminus (\cup_{i\notin J}E_i)$.
 We say that $\cX$ is {\em log Calabi-Yau} along $C$ if $C\cong \mathbb{P}^1_k$ and $C\setminus C^o$ consists of precisely two points, which we denote by $c_0$ and $c_{\infty}$.
  Denote by $0$ and $\infty$ be the unique elements of $I\setminus J$ such that $\{c_0\}=C\cap E_{0}$ and $\{c_{\infty}\}=C\cap E_{\infty}$ (note that $0$ and $\infty$ are not necessarily distinct).
   Then the fact that $\sum_{i\in I}N_i E_i$ is a principal divisor on $\cX$ implies that
  \begin{equation}\label{eq:principal}
  0=(\cX_k\cdot C)= N_{0}+N_{\infty}-\sum_{j\in J}b_jN_j.
  \end{equation}

 \begin{prop}\label{prop:toric}
  Assume that $\cX_k$ is log Calabi-Yau along $C$ and that $b_j>0$ for all $j\in J$.
   Then $\cX$ is toric along $C$.
 \end{prop}
 \begin{proof}
 We will construct a regular toric $R$-scheme $\cY$ such that $\cY_k$ is a strict normal crossings divisor that has a stratum $D$ satisfying $\comp{\cX}{C}\cong \comp{\cY}{D}$.
  Let $\iota$ be the greatest common divisor of the multiplicities $N_{i}$ with $i\in J\cup\{0\}$.
   We choose lattice vectors $u_i,\,i\in J\cup\{0\}$ in $\Z^J$ with the following property: if we set
    $v_0=(u_0,N_0/\iota)$ and $v_j=(u_j,N_j/\iota)$ in $\Z^J\oplus \Z$, for all $j\in J$, then the set $\{v_i,\,i\in J\cup\{0\}\,\}$ is a basis for
   $\Z^J\oplus \Z$.
      Now we set $$v_{\infty}=-v_0+\sum_{j\in J}b_jv_j$$ in $\Z^J\oplus \Z$. Because of the relation \eqref{eq:principal}, the last coordinate of $v_{\infty}$ equals $N_{\infty}/\iota$.

  For every $i$ in $J\cup\{0,\infty\}$, let $\rho_i$ be the ray in $\R^J\times \R_{\geq 0}$ spanned by the primitive vector $v_i$.
  Consider the cones $\sigma_0$ and $\sigma_{\infty}$ spanned by the rays $\rho_j$, $j\in J$ and by $\rho_0$ and $\rho_{\infty}$, respectively.
   The intersection of these cones is the common face spanned by the rays $\rho_j$, $j\in J$.
  Let $\Sigma$ be the fan in $\R^J\times \R_{\geq 0}$ with maximal cones $\sigma_0$ and $\sigma_\infty$.
     Then $\Sigma$ defines a toric $k$-variety $Y$. We consider the toric morphism $$Y\to \A^1_k=\Spec k[t]$$ associated with the morphism of cocharacter modules
   $$\Z^J\oplus \Z\mapsto \Z\colon (u,v)\mapsto \iota\cdot v,$$ and we set $\cY=Y\times_{k[t]}R$.

    The scheme $\cY$ is regular because the cones $\sigma_0$ and $\sigma_{\infty}$ are simple. Moreover, $\cY_k$ is a strict normal crossings divisor whose prime components correspond to the rays of $\Sigma$, with multiplicities given by $\iota$ times the last coordinates of the primitive generators of the rays; thus we can write
    $$\cY_k=\sum_{j\in J}N_j F_j + N_{0}F_{0}+N_{\infty}F_{\infty}.$$ Set $D=\cap_{j\in J}F_j$ and write $d_0,\,d_{\infty}$ for the intersection points of $D$ with $F_0$ and $F_{\infty}$, respectively.
 By \cite[\S5.1]{fulton}, we have $D\cdot F_j=-b_j$ for every $j\in J$.

  We will now construct an isomorphism of formal $R$-schemes $$f\colon \comp{\cX}{C}\to \comp{\cY}{D}.$$
  For every $n\geq 0$, we denote by $(\cX/C)_n$ the degree $n$ thickening of $C$ in $\cX$, that is, the closed subscheme of $\cX$ defined by the $(n+1)$-th power of the defining ideal of $C$. Thus $(\cX/C)_0=C$ and, by definition, $\comp{\cX}{C}$ is the direct limit of the schemes $(\cX/C)_n$ in the category of locally topologically ringed spaces.

 If $E_0$ and $E_\infty$ are distinct, then, for every $j\in J$, we denote by $\mathcal{L}_j$ the line bundle on $\comp{\cX}{C}$ induced by $\mathcal{O}_{\cX}(-E_j-b_jE_{\infty})$. This definition does not give the desired result when $E_0=E_\infty$. To include that case, we
 consider the formal completion of $E_{\infty}$ at $c_{\infty}$. This is a closed formal subscheme of $\comp{\cX}{C}$; we write $\mathcal{I}_{\infty}$ for its defining ideal sheaf, which is a principal ideal sheaf on $\comp{\cX}{C}$. For every $j\in J$, we denote by $\mathcal{L}'_j$ the line bundle on $\comp{\cX}{C}$ induced by $\mathcal{O}_{\cX}(-E_j)$,
  and we set $\mathcal{L}_j=\mathcal{L}'_j\otimes \mathcal{I}_{\infty}^{b_j}$. Then $\mathcal{L}_j$ is a line bundle on $\comp{\cX}{C}$, and our definition agrees with the previous one in the case where $E_0$ and $E_\infty$ are distinct.

 Since the restriction of $\mathcal{L}_j$ to $C\cong \mathbb{P}^1_k$ has degree $0$, we can choose a non-zero global section $s_j$ of $\mathcal{L}_j|_C$.
  The conormal bundle of $C$ in $\cX$ is given by
  $$\bigoplus_{j\in J}\mathcal{O}_C(-E_j)$$ which is a direct sum of ample line bundles, by our assumption that the numbers $b_j$ are all positive.
  This implies that the degree one cohomology of the conormal line bundle vanishes, so that the maps
  $$H^0((\cX/C)_{n+1},\mathcal{L}_j)\to H^0((\cX/C)_{n},\mathcal{L}_j) $$ are surjective for all $n\geq 0$.
   Thus we can lift $s_j$ to a global section of $\mathcal{L}_j$ on $\comp{\cX}{C}$, which we will still denote by $s_j$.
  The same argument produces a nowhere vanishing section $s_0$ of $\mathcal{O}_{\cX}(E_{\infty}-E_0)$ on $\comp{\cX}{C}$; its inverse $s_{\infty}=1/s_0$ is a nowhere vanishing global section
  of $\mathcal{O}_{\cX}(E_0-E_{\infty})$ on $\comp{\cX}{C}$.

  Consider the open formal subschemes
  $$\fX_0=\comp{\cX}{C}\setminus \{c_{\infty}\},\ \fX_{\infty}=\comp{\cX}{C}\setminus \{c_{0}\}, \quad \fY_0=\comp{\cY}{D}\setminus \{d_{\infty}\},\ \fY_\infty=\comp{\cY}{D}\setminus \{d_{0}\}$$ of $\comp{\cX}{C}$ and $\comp{\cY}{D}$.
   Note that $s_i$ is a global equation for $E_i$ on $\fX_0$, for every $i\in J\cup\{0\}$. Likewise, $s_{\infty}$ defines $E_{\infty}$ on $\fX_{\infty}$, and $s_{j}s^{-b_j}_{\infty}$ defines $E_j$ on $\fX_{\infty}$, for every $j\in J$.
    Moreover, $w'=ts_0^{-N_0}\prod_{j\in J}s_j^{-N_j}$ is an invertible regular function on $\comp{\cX}{C}$.
 Since $C$ is proper, $w'$ is constant on $C$, and, in particular, it has a $\iota$-th root; Hensel's lemma then implies that we can
 find a regular function $w$ on $\comp{\cX}{C}$ such that $w'=w^\iota$.

  Let $\{v_0^{\vee},v_j^{\vee}\,(j\in J)\}$ be the dual basis of $\{v_0,v_j\,(j\in J)\}$. Then we have
  $$ \fY_0=\Spf R\{\chi^{v^{\vee}_0}\}\llbr \chi^{v_j^{\vee}}\,(j\in J) \rrbr/(t- \prod_{i\in J\cup\{0\}}\chi^{N_i v_i^{\vee}} ).$$
   Choose integers $\alpha_0$ and $\alpha_j,\,j\in J$ such that $\alpha_0N_0+\sum_{j\in J}\alpha_jN_j=\iota$.
 Let $f_0\colon \fX_0\to \fY_0$ be the morphism of formal $R$-schemes defined by the morphism of topological $R$-algebras
 $$\mathcal{O}(\fY_0)\to \mathcal{O}(\fX_0)\colon  \chi^{v^{\vee}_i} \mapsto w^{\alpha_i}s_i, \mbox{ for all }i\in J\cup\{0\}. $$
 Let $\mathscr{J}$ be the largest ideal of definition on $\fY_0$. Then $\mathscr{J}(\fY_0)$ is generated by $\chi^{v_j^{\vee}},\,j\in J$.
  The ideal $\mathscr{J}\mathcal{O}_{\fX_0}$ is the largest ideal of definition on $\fX_0$, and its global sections are generated by
  $s_j,\,j\in J$. In particular, $f_0$ is adic.
  The morphism
 $$(f_0)_{\red}\colon C\setminus \{c_{\infty}\}=(\fX_0)_{\red}  \to   (\fY_0)_{\red}=D\setminus \{d_{\infty}\}$$ is an isomorphism.
 It follows from \cite[4.8.10]{ega3.1} that $f_0$ is a closed immersion; since $\fX_0$ and $\fY_0$ has the same dimension and $\fY_0$ is integral, $f_0$ is an isomorphism.

 Finally, we consider the second pair of affine charts  $\fX_{\infty},\,\fY_{\infty}$.
  The lattice vectors $\{-v_0^{\vee},v_j^{\vee}+b_jv^{\vee}_0\,(j\in J)\}$ form the dual basis of $\{v_{\infty},v_j\,(j\in J)\}$, and
    $$ \fY_{\infty}=\Spf R\{\chi^{-v^{\vee}_{0}}\}\llbr \chi^{v_j^{\vee}+b_jv^{\vee}_0}\,(j\in J) \rrbr/(t- \prod_{i\in J\cup\{0\}}\chi^{N_i v_i^{\vee}} ).$$
     Let $f_\infty\colon \fX_\infty\to \fY_\infty$ be the morphism of formal $R$-schemes defined by the morphism of topological $R$-algebras
 $\mathcal{O}(\fY_\infty)\to \mathcal{O}(\fX_\infty)$ that maps  $\chi^{-v^{\vee}_0}$ to $w^{-\alpha_0}s_{\infty}$ and $\chi^{v_j^{\vee}+b_jv^{\vee}_0}$ to $w^{\alpha_j+b_j\alpha_0}s_js^{-b_j}_{\infty}$, for all $j$ in $J$.
 By the same reasoning as above, one sees that $f_{\infty}$ is an isomorphism. By construction, it agrees with $f_0$ on the intersection of
 $\fX_0$ and $\fX_{\infty}$, and the isomorphisms $f_0$ and $f_{\infty}$ glue to an isomorphism of formal $R$-schemes
 $$f\colon \comp{\cX}{C}\to \comp{\cY}{D}.$$
 \end{proof}

\section{The smooth locus of the SYZ fibration}
\begin{theorem}\label{thm:main}
Let $X$ be a maximally degenerate projective Calabi-Yau variety over $K$ of dimension $n$, and assume that $X$ has a good minimal dlt-model $\cX$ over $R$ with reduced special fiber.
  Let $Z$ be the union of the faces of codimension $\geq 2$ in $\Delta(\cX^{\snc})= \Sk(X)$.
 Then the non-archimedean SYZ fibration $$\rho_{\cX}\colon X^{\an}\to \Sk(X)$$ associated with $\cX$ is
 an $n$-dimensional affinoid torus  fibration over $\Sk(X)\setminus Z$.
  Moreover, the induced integral affine structure on $\Sk(X)\setminus Z$ is compatible with the canonical piecewise integral affine structure on $\Sk(X)$ (see \eqref{sss:intaff}), in the sense that they give rise to the same piecewise integral affine functions on $\Sk(X)\setminus Z$.
\end{theorem}
Recall that such a model $\cX$ can always be found after a finite extension of the base field $K$ if $X$ is defined over a curve (Theorem \ref{thm:exist}), which is the most relevant case for applications to mirror symmetry (see Remark \ref{rem:curve}). We also recall that the essential skeleton $\Sk(X)$ is an $n$-dimensional closed pseudomanifold, by \cite[4.2.4]{NiXu}.  If $h^{i,0}(X)=0$ for $0<i<n$, then $\Sk(X)$ has the rational homology of the $n$-sphere $S^n$ \cite[4.2.4]{NiXu}. If, moreover, $X$ has dimension $3$ and trivial geometric fundamental group, then $\Sk(X)$ is homeomorphic to $S^3$ by \cite[\S34]{KoXu}; see Proposition \ref{prop:approx0}.

\begin{proof}
 We start by showing that $\rho_{\cX}$ is an affinoid torus fibration over the $n$-dimensional open faces of $\Delta(\cX^{\snc})$. Let $\mathring{\sigma}$ be such an open face; it corresponds to a $0$-dimensional stratum $\{x\}$ in $\cX_k$. It follows directly from the construction of $\rho_{\cX}$ that $\rho_{\cX}^{-1}(\mathring{\sigma})$ is the generic fiber of $\Spf \widehat{\mathcal{O}}_{\cX,x}$, the formal completion of $\cX$ at $x$, and that the restriction of $\rho_{\cX}$
    over $\mathring{\sigma}$ only depends on the formal $R$-scheme $\Spf \widehat{\mathcal{O}}_{\cX,x}$. By the definition of a dlt-model, $\cX$ is regular at $x$, and $\cX_k$ is a strict normal crossings divisor locally around $x$. By our assumption that $\cX_k$ is reduced, we know that the $R$-algebra $\widehat{\mathcal{O}}_{\cX,x}$ is isomorphic to
    $$R\llbr z_0,\ldots,z_n \rrbr/(t-z_0\cdot \ldots \cdot z_n).$$
 Thus we can identify the restriction of $\rho_{\cX}$ over $\mathring{\sigma}$ with the restriction of the tropicalization map
 $$\rho_{\mathbb{G}^n_{m,K}}\colon \mathbb{G}^{n,\an}_{m,K}\to \R^n$$
 over the standard $n$-dimensional open simplex $$\{u\in (\R_{>0})^{n+1}\,|\,u_0+\ldots+u_n=1\}.$$
 It follows that $\rho_{\cX}$ is an $n$-dimensional affinoid torus fibration over $\mathring{\sigma}$, and that the induced integral affine structure on $\mathring{\sigma}$ is compatible with the canonical piecewise integral affine structure on $\Sk(X)$.

 If we can extend the integral affine structure from the union of the $n$-dimensional open faces to $\Sk(X)\setminus Z$, then the result will automatically be compatible with the canonical piecewise integral affine structure on $\Sk(X)$. Indeed, the essential skeleton $\Sk(X)$ is purely of dimension $n$ by \cite[4.2.4]{NiXu}.
  Therefore, a real-valued function on a subset of $\Sk(X)$ is piecewise integral affine if and only if it is continuous and extends to a real-valued function on an open neighbourhood of its domain whose restriction to the interior of each $n$-dimensional face of $\Sk(X)$ is piecewise integral affine.

 Now, we prove that $\rho_{\cX}$ is an affinoid torus fibration locally around the open faces of codimension one in $\Delta(\cX^{\snc})$.
  We fix such a face $\mathring{\tau}$, corresponding to a $1$-dimensional stratum $C$ in $\cX_k$.
By Corollary \ref{cor:dlt}, the model $\cX$ is snc along $C$. By means of a finite sequence of blow-ups at zero-dimensional strata, we can moreover arrange that, for every prime component $E$ of $\cX_k$ that contains $C$, the intersection number $(C\cdot E)$ is negative. This may destroy the property that $\cX_k$ is reduced, but it preserves the properties that $\cX$ is snc along every one-dimensional stratum, $\cX$ is a good minimal dlt-model, and $\cX$ satisfies assumption \eqref{sss:assum}. Moreover, the sequence of blow-ups has no effect on the map $\rho_{\cX}$, by \cite[3.1.7]{MuNi}. The effect on the skeleton $\Delta(\cX^{\snc})$ is a sequence of star subdivisions of the $n$-dimensional faces corresponding to the zero-dimensional strata that are blown up \cite[3.1.9]{MuNi}. This does not affect the face $\mathring{\tau}$.

 Thus it suffices to prove that $\rho_{\cX}$ is an affinoid torus fibration over an open neighbourhood of $\mathring{\tau}$, under the following alternative assumptions on the model $\cX$ and the $1$-dimensional stratum  $C$:
  \begin{itemize}
  \item  $\cX$ is a good minimal dlt-model satisfying \eqref{sss:assum};
  \item  the model $\cX$ is snc along $C$;
   \item for every prime component $E$ of $\cX_k$ that contains $C$, the component $E$ has multiplicity one in $\cX_k$, and the intersection number $(C\cdot E)$ is negative.
 \end{itemize}
  We will prove that $\rho_{\cX}$ is an $n$-dimensional affinoid torus  fibration over the open star of $\mathring{\tau}$ in $\Delta(\cX^{\snc})$ (that is, the union of $\mathring{\tau}$ with the two $n$-dimensional open faces whose closure contains $\mathring{\tau}$).

  By adjunction, the model $\cX$ is log Calabi-Yau along $C$ in the sense of \eqref{sss:CY}. We denote the $0$-dimensional strata
  contained in $C$ by $c_0$ and $c_\infty$.
  The model $\cX$ is toric along $C$, by Proposition \ref{prop:toric}.
   More precisely, the proof of Proposition \ref{prop:toric} gives an explicit description of the formal completion $\comp{\cX}{C}$ of $\cX$ along $C$.
   Note that, under our assumptions and with the notations in that proof, the number $\iota$ is equal to one and $N_j=1$ for every $j\in J$, so that we can make the construction of the fan $\Sigma$ more explicit:
   we choose a bijection of $J$ with $\{1,\ldots,n\}$. Then we can take for $(u_0,\ldots,u_{n-1})$ the standard basis of $\Z^n$, and set $u_n=0$.
   The vector $v_{\infty}$ is now given by $(-1,b_1,\ldots,b_{n-1},N_{\infty})$. Let $\Sigma$ be the fan with maximal cones $\sigma_0$ and $\sigma_{\infty}$.
   Then the toric scheme $\cY$ constructed in the proof of Proposition \ref{prop:toric} is precisely the torus embedding associated with $\Sigma$ in the sense of Example \ref{exam:toric}.

    Let $U$ be the union in $\Delta(\cX^{\snc})$ of the open faces corresponding to the strata $c_{0}$, $c_{\infty}$ and $C$ in $\cX_k$. This is an open neighbourhood of $\mathring{\tau}$ in $\Sk(X)$.
     Let $V$ be the interior of the intersection of $|\Sigma|$ with $\R^n\times \{1\}$.
    It follows directly from the construction of $\rho_{\cX}$ that $\rho_{\cX}^{-1}(U)$ is the generic fiber of $\comp{\cX}{C}$, and that the restriction of $\rho_{\cX}$
    over $U$ only depends on the formal $R$-scheme $\comp{\cX}{C}$. If $D$ is the torus orbit in $\cY_k$ corresponding to the codimension one cone $\sigma_0\cap \sigma_{\infty}$ in $\Sigma$, then we have shown in the proof of Proposition \ref{prop:toric} that $\comp{\cX}{C}$ is isomorphic to $\comp{\cY}{D}$.
     Thus, by Example \ref{exam:toric}, we can identify the restriction of $\rho_{\cX}$ over $U$ with the restriction of $\rho_{\mathbb{G}^n_{m,K}}$ over $V$, which is an $n$-dimensional affinoid torus  fibration by definition.
\end{proof}

\sss Note that the proof of Proposition \ref{prop:toric} gives an explicit description of the integral affine structure on $\Sk(X)\setminus Z$ induced
by the non-archimedean SYZ fibration: after our finite sequence of blow-ups at zero-dimensional strata, the gluing data along codimension one faces of the skeleton are determined by the intersection numbers $(C\cdot E)$. This is quite similar to the constructions for log Calabi-Yau surfaces in \cite{GHK,yu} and for toric degenerations in the Gross-Siebert program \cite{GS-toric}.


\begin{thebibliography}{XXXXX11}
\bibitem[SGA1]{sga1}
Rev{\^e}tements {\'e}tales et groupe fondamental. \newblock S\'eminaire de G\'eom\'etrie Alg\'ebrique du Bois-Marie 1960--61 (SGA 1).
 Dirig\'e par A.~Grothendieck. Augment{\'e} de deux expos{\'e}s de M.~Raynaud. Updated and annotated reprint of the 1971 original [Lecture Notes in Math., 224, Springer, Berlin]. Volume~3 of {Documents Math{\'e}matiques (Paris).} Soci\'et\'e Math\'ematique de France, Paris, 2003.



\bibitem[Be90]{berk-book}
V.G.~Berkovich.
\newblock  Spectral theory and analytic geometry over non-Archimedean fields.
\newblock Volume~33 of
{\em Mathematical Surveys and Monographs.} American Mathematical Society, Providence, RI, 1990.


\bibitem[Be96]{berk-vanish2}
V.~Berkovich.
\newblock {Vanishing cycles for formal schemes. II.}
\newblock {\em Invent. Math.} 125(2):367--390, 1996.

\bibitem[BJ17]{BoJo}
S.~Boucksom and M.~Jonsson.
\newblock Tropical and non-Archimedean limits of degenerating families of volume forms.
\newblock {\em J. {\'E}c. Polytech. Math.} 4:87--139, 2017.

\bibitem[COGP91]{COGP}
P.~Candelas, X.~de la Ossa, P.~Green and L.~Parkes.
\newblock A pair of Calabi-Yau manifolds as an exactly soluble superconformal theory.
\newblock {\em Nuclear Phys. B} 359(1):21--74, 1991.

\bibitem[De68]{deligne}
P.~Deligne.
\newblock {Th{\'e}or{\`e}me de Lefschetz et crit{\`e}res de d{\'e}g{\'e}n{\'e}rescence de suites spectrales.}
\newblock {\em Inst. Hautes {\'E}tudes Sci. Publ. Math.} 35:259--278, 1968.

\bibitem[FM83]{friedman-morrison}
R.~Friedman and D.~Morrison,
\newblock  The Birational
Geometry of Degenerations.
\newblock  Volume~29 of {\em Progress in Mathematics}. Birkh{\"a}user,  Boston, Mass., 1983.

\bibitem[Fu93]{fulton}
W.~Fulton.
\newblock   Introduction to toric varieties.
\newblock Volume~131 of {\em Annals of Mathematics Studies.} Princeton University Press, Princeton, NJ, 1993.

\bibitem[GW13]{GW-AG}
U.~G{\"o}rtz and T.~Wedhorn.
\newblock {Algebraic geometry I.}
\newblock {\em Advanced Lectures in Mathematics.} Vieweg + Teubner, Wiesbaden, 2010.

\bibitem[Gr13]{Gro}
M.~Gross.
\newblock  Mirror symmetry and the Strominger-Yau-Zaslow
conjecture.
\newblock  In: {\em Current developments in mathematics 2012}, pages 133--191. Int. Press, Somerville, MA, 2013.

\bibitem[GHK15]{GHK}
M.Gross, P.~Hacking, and S.~Keel.
\newblock {Mirror symmetry for log Calabi-Yau surfaces I.}
\newblock {\em Inst. Hautes {\'E}tudes Sci. Publ. Math.} 155:65--168, 2015.

\bibitem[GS11a]{GS}
M.~Gross and B.~Siebert.
\newblock  From real affine geometry to complex geometry.
\newblock {\em Ann. of Math. (2)}, 174(3):1301--1428, 2011.

\bibitem[GS11b]{GS-toric}
M.~Gross and B.~Siebert.
\newblock An invitation to toric degenerations.
\newblock In:  Geometry of special holonomy and related topics. Volume~16 of {\em Surv.
Differ. Geom.}, pages 43--78. Int. Press, Somerville, MA, 2011.



\bibitem[GW00]{GW}
M.~Gross and P.M.~Wilson.
\newblock Large complex structure limits of $K3$ surfaces.
\newblock {\em J. Differential Geom.} 55(3):475--546, 2000.

\bibitem[EGA3.1]{ega3.1}
A.~Grothendieck.
\newblock { \'El\'ements de g\'eom\'etrie alg\'ebrique. III. \'Etude cohomologique des faisceaux coh\'erents. I.}
\newblock  {\em Inst. Hautes \'Etudes Sci. Publ. Math.} 11:5--167, 1961.

\bibitem[HN17]{HaNi}
L.H.~Halle and J.~Nicaise.
\newblock Motivic zeta functions of degenerating Calabi-Yau varieties.
\newblock {\em Math. Ann.} 370(3):1277--1320, 2018.

\bibitem[Ka94]{kawamata}
Y.~Kawamata.
\newblock  Semistable minimal models of threefolds in positive or mixed characteristic.
\newblock {\em J. Algebraic Geom.} 3(3):463--491, 1994.

\bibitem[KY18]{KY18}
S.~Keel and T.Y.~Yu.
\newblock {The {F}robenius structure conjecture in dimension two.}
\newblock In preparation, 2018.


\bibitem[KKMS73]{KKMS}
G.~Kempf, F.F.~Knudsen, D.~Mumford and B.~Saint-Donat.
\newblock {\em Toroidal embeddings. {I.}}
\newblock Volume~339 of {\em Lecture Notes in Mathematics.} Springer-Verlag, Berlin-New York, 1973.





\bibitem[Ko13]{kollar}
J.~Koll{\'a}r.
\newblock  Singularities of the minimal model program. With a collaboration of S{\'a}ndor Kov{\'a}cs.
\newblock Volume~200 of {\em Cambridge Tracts in Mathematics.} Cambridge University Press, Cambridge, 2013.

\bibitem[KM98]{kollar-mori}
J.~Koll{\'a}r and S.~Mori.
 \newblock  Birational geometry of algebraic varieties.
   \newblock Volume~134 of {\em Cambridge Tracts in Mathematics.} Cambridge University Press, Cambridge, 1998.

\bibitem[KNX18]{KNX}
J.~Koll\'ar, J.~Nicaise and C.~Xu.
\newblock Semi-stable extensions over 1-dimensional bases.
\newblock {\em Acta Math. Sinica} 34(1):103--113, 2018.

\bibitem[KX16]{KoXu}
J.~Koll\'ar and C.~Xu.
\newblock
{The dual complex of Calabi-Yau pairs.}
\newblock {\em Invent. Math.}, 205(3):527--557, 2016.



\bibitem[KS00]{KoSo2000}
M.~Kontsevich and Y.~Soibelman.
\newblock Homological mirror symmetry and torus fibrations.
\newblock In: {\em Symplectic geometry and mirror symmetry (Seoul, 2000)}, pages 203--263. World Sci. Publ., River Edge, NJ, 2001.

\bibitem[KS06]{KoSo}
M.~Kontsevich and Y.~Soibelman.
\newblock   Affine structures and
non-archimedean analytic spaces.
\newblock In: P.~Etingof, V.~Retakh and I.M.~Singer (eds). {\em The unity of
mathematics. In honor of the ninetieth birthday of I. M. Gelfand.}
Volume~244 of {\em Progress in Mathematics}, pages  312--385. Birkh\"{a}user
Boston, Inc., Boston, MA, 2006.

\bibitem[K{\"u}98]{kunnemann}
K.~K{\"u}nnemann.
\newblock Projective regular models for abelian varieties, semistable reduction, and the height pairing.
\newblock {\em Duke Math. J.}, 95(1):161--212, 1998.


\bibitem[Mu72]{mumford}
D.~Mumford.
\newblock An analytic construction of degenerating abelian varieties over complete rings.
\newblock {\em Compositio Math.} 24:239--272, 1972.

\bibitem[MN15]{MuNi}
M.~Musta\c{t}\u{a} and J.~Nicaise.
\newblock Weight functions on non-archimedean analytic spaces and the
Kontsevich-Soibelman skeleton.
\newblock {\em Alg.~Geom.} 2(3):365--404, 2015.

\bibitem[Na98]{nakayama}
C.~Nakayama.
\newblock Nearby cycles for log smooth families.
\newblock {\em Compositio Math.} 112(1):45--75, 1998.

\bibitem[NX16a]{NiXu}
J.~Nicaise and C.~Xu.
\newblock The essential skeleton of a degeneration of algebraic varieties.
\newblock  {\em Amer.~Math.~J.},  138(6):1645--1667, 2016.

 \bibitem[NX16b]{NiXu2}
J.~Nicaise and C.~Xu.
\newblock Poles of maximal order of motivic zeta functions.
\newblock {\em Duke Math. J.} 165(2):217--243, 2016.

\bibitem[SYZ96]{SYZ}
A.~Strominger, S.-T.~Yau and E.~Zaslow.
\newblock {\em {Mirror Symmetry is T-duality.}}
\newblock {\em Nuclear Phys. B}, 479(1-2):243--259, 1996.

\bibitem[Yu16a]{yu}
T.Y.~Yu.
\newblock {Enumeration of holomorphic cylinders in log Calabi-Yau surfaces I.}
\newblock {\em Math. Ann.}, 366(3):1649--1675, 2016.

\bibitem[Yu16b]{yub}
T.Y.~Yu.
\newblock{Enumeration of holomorphic cylinders in log {C}alabi-{Y}au surfaces. {II}. {P}ositivity, integrality and the gluing formula}.
\newblock  Preprint, arXiv:1608.07651.



\end{thebibliography}
\end{document}